\documentclass[10pt,a4paper]{article}

\usepackage{graphicx}
\usepackage{color}
\usepackage{indentfirst}
\usepackage{amsmath,amssymb,bm}
\usepackage{cases}
\usepackage{latexsym,bm,amsmath,amssymb,amsthm}

\begin{document}

\title{\bfseries  \Large Killing Vector Fields on Multiply Warped Products with a Semi-symmetric Metric Connection}
\author{\normalsize Quan Qu
\thanks{Corresponding author.
\newline \mbox{} \quad\, E-mail addresses: quq453@nenu.edu.cn.
\newline  \mbox{} \quad\,   School of Mathematics and Statistics, Northeast Normal University, Changchun Jilin, 130024, China}}\date{}
\maketitle

\begin{abstract}
In this paper, we define a semi-symmetric metric Killing vector field, then study semi-symmetric metric Killing vector fields on warped and multiply warped products with a semi-symmetric metric connection. We also study Killing and 2-Killing vector fields on multiply warped products.
\end{abstract}
\paragraph{} \mbox{}\quad \, \normalsize 2010 Mathematics Subject Classification. 53B05, 53C21, 53C25, 53C50, 53C80.
\section{\large Introduction}
The warped product $M_{1} \times _{f}M_{2}$ of two pseudo-Riemannian manifolds $(M_{1},g_{1})$ and $(M_{2},g_{2})$ with a smooth function $f : M_{1} \to(0,\infty)$ is a product manifold of form $M_{1} \times M_{2}$ with the metric tensor $g = g_{1}\oplus f^2g_{2}$. Here, $(M_{1},g_{1})$ is called the base manifold and $(M_{2},g_{2})$ is called as the fiber manifold and $f$ is called as the warping function. The concept of warped products was first introduced by Bishop and O'Neill in [1] to construct examples of Riemannian manifolds with negative curvature.  \par
One can generalize warped products to multiply warped products. A multiply warped product $(M,g)$ is the product manifold $M=B \times_{f_{1}}M_{1} \times_{f_{2}}M_{2}\times\cdots \times_{f_{m}}M_{m}$ with the metric $g=g_{B}\oplus f_{1}^2g_{M_{1}}\oplus f_{2}^2g_{M_{2}}\cdots\oplus f_{m}^2g_{M_{m}}$, where for each $i\in\{1,\cdots,m\},f_{i}:B \to(0,\infty)$ is smooth and $(M_{i},g_{M_{i}})$ is a pseudo-Riemannian manifold.\par
Killing vector fields have been studied on Riemannian and pseudo-Riemannian manifolds for a long time. The problems of existence and characterization of Killing vector fields are important and widely discussed by both mathematicians and physicists. In [2], the concept of 2-Killing vector fields, as a new generalization of Killing vector fields, was first introduced and studied on Riemannian manifolds. In [3], Shenawy and \"Unal studied 2-Killing vector fields on warped product manifolds, and applied their result to standard static space-times.\par
The definition of a semi-symmetric metric connection was given by H. Hayden in [4]. In 1970, K. Yano in [5] considered a semi-symmetric metric connection and studied some of its properties. Motivated by the Yano's result, in [6], Sular and \"Ozg\"ur studied warped product manifolds with a semi-symmetric metric connection. In [7], professor Yong Wang studied multiply warped products with a semi-symmetric metric connection.

This paper is arranged as follows. In Section 2, we introduce a semi-symmetric metric connection, and then recall the definition of warped products and multiply warped products. In Section 3, we define a semi-symmetric metric Killing vector field, then study semi-symmetric metric Killing vector fields on warped products with a semi-symmetric metric connection. In Section 4, we study semi-symmetric metric Killing vector fields on multiply warped products with a semi-symmetric metric connection. In Section 5, we study Killing vector fields on multiply warped products. Finally in Section 6, we study 2-Killing vector fields on multiply warped products.
\section{\large Preliminaries}\label{sec2}
Let $ M $ be a Riemannian manifold with Riemannian metric $g$. A linear connection $\overline\nabla$ on a Riemannian manifold $M$ is called a semi-symmetric connection if the torsion tensor $T$ of the connection $\overline\nabla$
$$T(X, Y ) = \overline\nabla_{X}Y-\overline\nabla_{Y}X-[X,Y] \eqno{(1)}$$
satisfies  $$T(X, Y ) = \pi(Y)X-\pi(X)Y  \eqno{(2)}$$
where $\pi$ is a 1-form associated with the vector field $P$ on $M$ defined by $\pi(X)=g(X,P)$. $\overline\nabla$ is called a semi-symmetric metric connection if it satisfies $\overline\nabla g=0$.

If $\nabla$ is the Levi-Civita connection of $M$, the semi-symmetric metric connection $\overline\nabla$ is given by
$$\overline\nabla_{X}Y=\nabla_{X}Y+\pi(Y)X-g(X,Y)P, \eqno{(3)}$$
(see [5]).\par
The $(1,3)$ type curvature $R$ is defined by $R(X,Y)=\nabla_{X}\nabla_{Y}-\nabla_{Y}\nabla_{X}-\nabla_{[X,Y]},$ and
$(0,4)$ type curvature is defined by $R(X,Y,Z,W)=g(R(X,Y)Z,W).$

We recall the definition of warped product and multiply warped product as follows:
\newtheorem{Definition}{Definition}[section]
\begin{Definition}
Let $(M_{1},g_{1})$ and $(M_{2},g_{2})$ be two pseudo-Riemannian manifolds,$\;f:M_{1} \to(0,\infty)$ be a smooth function. The {\bfseries warped product} $M=M_{1}\times _{f}M_{2}$ is the product manifold $M_{1} \times M_{2}$ with the metric tensor $g = g_{1}\oplus f^2g_{2}$. The function $f$ is called the warping function of the warped product.
\end{Definition}

\begin{Definition}
Let $(B,g_{B})$ and $(M_{i},g_{i})$ be pseudo-Riemannian manifolds,$\;f_{i}:B \to(0,\infty)$ be smooth functions, where for each $i\in\{1,\cdots,m\}$. The {\bfseries multiply warped product} $M=B\times _{f_{1}}M_{1}\times _{f_{2}}M_{2}\times\cdots\times _{f_{m}}M_{m}$ is the product manifold $B\times M_{1}\times M_{2}\times\cdots\times M_{m}$ with the metric tensor $g = g_{B}\oplus f_{1}^2g_{1}\oplus f_{2}^2g_{2}\oplus\cdots\oplus f_{m}^2g_{m}$. The function $f_{i}$ is called the warping function of the multiply warped product.
\end{Definition}

\newtheorem{Lemma}[Definition]{Lemma}

\section{\large Semi-symmetric Metric Killing Vector Fields on Warped Products with a Semi-symmetric Metric Connection}
\begin{Lemma}\label{l1}$^{[6]}$
Let $M=M_{1}\times _{f}M_{2}$ be a warped product with $\overline\nabla$. If $X_{1},Y_{1}\in\Gamma(TM_{1}),\;X_{2},Y_{2}\in\Gamma(TM_{2})$ and $P\in\Gamma(TM_{1}),$ then\\
$(1)\overline \nabla_{X_{1}}Y_{1}=\overline \nabla_{X_{1}}^{1}Y_{1}; $  \\
$(2)\overline \nabla_{X_{1}}Y_{2}=\frac{X_{1}f}{f}Y_{2};$   \\
$(3)\overline \nabla_{Y_{2}}X_{1}=\Big[\frac{X_{1}f}{f}+\pi(X_{1})\Big]Y_{2}; $   \\
$(4)\overline \nabla_{X_{2}}Y_{2}=-fg_{2}(X_{2},Y_{2})\mathrm{grad}f+\nabla_{X_{2}}^{2}Y_{2}-f^{2}g_{2}(X_{2},Y_{2})P.$
\end{Lemma}

\begin{Lemma}\label{l1}$^{[6]}$
Let $M=M_{1}\times _{f}M_{2}$ be a warped product with $\overline\nabla$. If $X_{1},Y_{1}\in\Gamma(TM_{1}),\;X_{2},Y_{2}\in\Gamma(TM_{2})$ and $P\in\Gamma(TM_{2}),$ then\\
$(1)\overline \nabla_{X_{1}}Y_{1}=\nabla_{X_{1}}^{1}Y_{1}-g_{1}(X_{1},Y_{1})P; $  \\
$(2)\overline \nabla_{X_{1}}Y_{2}=\frac{X_{1}f}{f}Y_{2}+\pi(Y_{2})X_{1};$   \\
$(3)\overline \nabla_{Y_{2}}X_{1}=\frac{X_{1}f}{f}Y_{2}; $   \\
$(4)\overline \nabla_{X_{2}}Y_{2}=
-fg_{2}(X_{2},Y_{2})\mathrm{grad}f+\nabla_{X_{2}}^{2}Y_{2}+\pi(Y_{2})X_{2}-f^{2}g_{2}(X_{2},Y_{2})P.$
\end{Lemma}

\begin{Lemma}
Let $M$ be a pseudo-Riemannian manifold, $\zeta\in\Gamma(TM),\,\mathcal{L}$ is the Lie derivative on $M$ with respect to $\zeta,$ then:
$$(\mathcal{L}_{\zeta}g)(X,Y)=g(\nabla_{X}\zeta,Y)+g(\nabla_{Y}\zeta,X)  \eqno{(4)} $$
for any vector fields $X,Y\in\Gamma(TM).$
\end{Lemma}

\begin{Definition}
Let $(M,g)$ be a pseudo-Riemannian manifold, $\zeta\in\Gamma(TM),\,\mathcal{L}$ is the Lie derivative on $M$ with respect to $\zeta,$ then $\zeta$ is called {\bfseries a Killing vector field} if
$$\mathcal{L}_{\zeta}g=0  \eqno{(5)} $$
\end{Definition}

Now we define the semi-symmetric metric Lie derivative as follows:
\begin{Definition}
Let $M$ be a pseudo-Riemannian manifold, $\zeta\in\Gamma(TM),\,\overline{\mathcal{L}}$ is called {\bfseries semi-symmetric metric Lie derivative} on $M$ with respect to $\zeta,$ if it satisfied:
$$(\overline{\mathcal{L}}_{\zeta}g)(X,Y)=g(\overline\nabla_{X}\zeta,Y)+g(\overline\nabla_{Y}\zeta,X)  \eqno{(6)} $$
for any vector fields $X,Y\in\Gamma(TM).$
\end{Definition}

\begin{Definition}
Let $(M,g)$ be a pseudo-Riemannian manifold, $\zeta\in\Gamma(TM),\,\overline{\mathcal{L}}$ is the semi-symmetric metric Lie derivative on $M$ with respect to $\zeta,$ then $\zeta$ is called {\bfseries a semi-symmetric metric Killing vector field} if
$$\overline{\mathcal{L}}_{\zeta}g=0  \eqno{(7)} $$
\end{Definition}

Using the symmetry in the equation $(4),$ we can get:
\begin{Lemma}
If $(M,g,\nabla)$ is a pseudo-Riemannian manifold, then $\zeta\in\Gamma(TM)$ is a Killing vector field if and only if:
$$g(\nabla_{X}\zeta,X)=0      \eqno{(8)}$$
\end{Lemma}
Using the symmetry in the equation $(6),$ we can get:
\begin{Lemma}
If $(M,g,\overline{\nabla})$ is a pseudo-Riemannian manifold, then $\zeta\in\Gamma(TM)$ is a semi-symmetric metric Killing vector field if and only if:
$$g(\overline\nabla_{X}\zeta,X)=0      \eqno{(9)}$$
\end{Lemma}

\newtheorem{Remark}[Definition]{Remark}
\begin{Remark}
The relationship between semi-symmetric metric Killing vector field and Killing vector field:
$$g(\overline\nabla_{X}\zeta,X)=g\big(\nabla_{X}\zeta+\pi(\zeta)X-g(X,\zeta)P,X\big)
=g(\nabla_{X}\zeta,X)+\pi(\zeta)\|X\|^2-\pi(X)g(X,\zeta).$$
So when $\pi(\zeta)\|X\|^2=\pi(X)g(X,\zeta),$ we have
$g(\overline\nabla_{X}\zeta,X)=0 \Leftrightarrow g(\nabla_{X}\zeta,X)=0.$ So we get:
\end{Remark}

\newtheorem{Proposition}[Definition]{Proposition}
\begin{Proposition}
Let $(M,g,\overline{\nabla})$ be a pseudo-Riemannian manifold, $\zeta\in\Gamma(TM),$ if $\pi(\zeta)\|X\|^2=\pi(X)g(X,\zeta),$ then the following two conditions are equivalent:\\
$(1)\,\zeta$ is a semi-symmetric metric Killing vector field on $M$.\\
$(2)\,\zeta$ is a Killing vector field on $M$.\\
In particular, if $P=0,$ then the above two conditions are equivalent.
\end{Proposition}

\begin{Remark}
It is easy to see that when $P=0,$ the semi-symmetric metric Lie derivative becomes to Lie derivative, and the semi-symmetric metric Killing vector field becomes to Killing vector field.
\end{Remark}

\newtheorem{Example}[Definition]{Example}
\begin{Example}
The semi-symmetric metric Killing vector field on $I,$ where $I$ is an open interval in $\mathbb{R},$ and
$g_{I}=\mathrm{d}t^2.$ Suppose $\zeta,X\in\Gamma(TI),$ then we have $\zeta=u\partial t,\,X=v\partial t,$ where
$u\neq0,\,v\neq0,$ then:
$$\pi(\zeta)\|X\|^2=g(P,\zeta)g(X,X)=g(P,u\partial t)g(v\partial t,v\partial t)=uv^2g(P,\partial t),$$
$$\pi(X)g(X,\zeta)=g(P,X)g(X,\zeta)=g(P,v\partial t)g(u\partial t,v\partial t)=uv^2g(P,\partial t),$$
which means $\pi(\zeta)\|X\|^2=\pi(X)g(X,\zeta),$ so by Proposition 3.10, we have the semi-symmetric metric Killing vector field on $I$ is equivalent to the Killing vector field on $I,$ and
$$g(\overline\nabla_{X}\zeta,X)=g(\nabla_{X}\zeta,X)=g(\nabla_{v\partial t}u\partial t,v\partial t)
=g(v\nabla_{\partial t}u\partial t,v\partial t)=g(v\dot{u}\partial t,v\partial t)=\dot{u}v^2=0,$$
since $v\neq0,$ we have $\dot{u}=0,$ so $u=a\in\mathbb{R}\setminus\{0\}.$ \par
In a word, the semi-symmetric metric Killing vector field on $I$ and the Killing vector field on $I$ has the form $a\partial t,$ where $a\in\mathbb{R}\setminus\{0\}.$
\end{Example}

By Lemma 3.1 and straightforward computation, we can get:
\begin{Proposition}
Let $M=M_{1}\times _{f}M_{2}$ be a warped product with $\overline\nabla,\,\zeta\in\Gamma(TM),\,P\in\Gamma(TM_{1})$. Then for any vector fields $X,Y\in\Gamma(TM),$ we have:
\begin{eqnarray} \setcounter{equation}{10}
(\overline{\mathcal{L}}_{\zeta}g)(X,Y)&=&(\overline{\mathcal{L}}^{1}_{\zeta_{1}}g_{1})(X_{1},Y_{1})
+f^{2}(\mathcal{L}^{2}_{\zeta_{2}}g_{1})(X_{2},Y_{2})+2[f\zeta_{1}(f)+f^{2}\pi(\zeta_{1})]g_{2}(X_{2},Y_{2}) \nonumber \\
& &-f^{2}g_{2}(X_{2},\zeta_{2})\pi(Y_{1})-f^{2}g_{2}(Y_{2},\zeta_{2})\pi(X_{1})
\end{eqnarray}
\end{Proposition}

By Lemma 3.2 and straightforward computation, we can get:
\begin{Proposition}
Let $M=M_{1}\times _{f}M_{2}$ be a warped product with $\overline\nabla,\,\zeta\in\Gamma(TM),\,P\in\Gamma(TM_{2})$. Then for any vector fields $X,Y\in\Gamma(TM),$ we have:
\begin{eqnarray}
(\overline{\mathcal{L}}_{\zeta}g)(X,Y)&=&(\mathcal{L}^{1}_{\zeta_{1}}g_{1})(X_{1},Y_{1})
+f^{2}(\mathcal{L}^{2}_{\zeta_{2}}g_{2})(X_{2},Y_{2})+2\pi(\zeta_{2})g(X,Y)\nonumber \\
& &+2f\zeta_{1}(f)g_{2}(X_{2},Y_{2})-\pi(Y_{2})g(X,\zeta)-\pi(X_{2})g(Y,\zeta)
\end{eqnarray}
\end{Proposition}
Using the equation $(6)$ and Proposition 3.13, we have:
\newtheorem{Corollary}[Definition]{Corollary}
\begin{Corollary}
Let $M=M_{1}\times _{f}M_{2}$ be a warped product with $\overline\nabla,\,\zeta\in\Gamma(TM),\,P\in\Gamma(TM_{1})$. Then for any vector field $X\in\Gamma(TM),$ we have:
\begin{eqnarray}
g(\overline\nabla_{X}\zeta,X)&=&g_{1}(\overline\nabla^1_{X_{1}}\zeta_{1},X_{1})
+f^{2}g_{2}(\nabla^2_{X_{2}}\zeta_{2},X_{2}) \nonumber \\
&-&f^{2}\pi(X_{1})g_{2}(X_{2},\zeta_{2})
+\big[f\zeta_{1}(f)+f^{2}\pi(\zeta_{1})\big]\|X_{2}\|_{2}^2.
\end{eqnarray}
\end{Corollary}
Using the equation $(6)$ and Proposition 3.14, we have:
\begin{Corollary}
Let $M=M_{1}\times _{f}M_{2}$ be a warped product with $\overline\nabla,\,\zeta\in\Gamma(TM),\,P\in\Gamma(TM_{2})$. Then for any vector field $X\in\Gamma(TM),$ we have:
\begin{eqnarray}
g(\overline\nabla_{X}\zeta,X)&=&g_{1}(\nabla^1_{X_{1}}\zeta_{1},X_{1})
+f^{2}g_{2}(\nabla^2_{X_{2}}\zeta_{2},X_{2})+\pi(\zeta_{2})\|X\|^2  \nonumber \\
&+&f\zeta_{1}(f)\|X_{2}\|_{2}^2-\pi(X_{2})g(X,\zeta).
\end{eqnarray}
\end{Corollary}

By Corollary 3.15, we can easily get:
\begin{Proposition}
Let $M=M_{1}\times _{f}M_{2}$ be a warped product with $\overline\nabla,\,P\in\Gamma(TM_{1}),\,\zeta=\zeta_{1}+\zeta_{2}\in\Gamma(TM)$. Then $\zeta$ is a semi-symmetric metric Killing vector field if one of the following conditions holds:\\
$(1)\,\zeta=\zeta_{1},\,\zeta_{1}$ is a semi-symmetric metric Killing vector field, and $\zeta_{1}(f)+f\pi(\zeta_{1})=0.$\\
$(2)\,\zeta=\zeta_{2},\,\zeta_{2}$ is a  Killing vector field and $\pi(X_{1})g_{2}(X_{2},\zeta_{2})=0.$\\
$(3)\,\zeta=\zeta_{1}+\zeta_{2},\,\zeta_{1}$ is a semi-symmetric metric Killing vector field, $\zeta_{2}$ is a  Killing vector field, and \\ \mbox{}\quad\, $\zeta_{1}(f)+f\pi(\zeta_{1})=0,\,\pi(X_{1})g_{2}(X_{2},\zeta_{2})=0.$
\end{Proposition}

\begin{Proposition}
Let $M=M_{1}\times _{f}M_{2}$ be a warped product with $\overline\nabla,\,P\in\Gamma(TM_{1}),\,\zeta=\zeta_{1}+\zeta_{2}$ is a semi-symmetric metric Killing vector field, then:\\
$(1)\,\zeta_{1}$ is a semi-symmetric metric Killing vector field on $M_{1}.$\\
$(2)\,\zeta_{2}$ is a Killing vector field on $M_{2}$ if $\zeta_{1}(f)+f\pi(\zeta_{1})=0.$
\end{Proposition}

Now we recall the definition of generalized Robertson-Walker space-times as follow:
\begin{Definition}
A {\bfseries generalized Robertson-Walker space-time} $(M,g)$ is a warped product of the form $M=I\times_{f}M_{2}$ with the metric tensor $g=-\mathrm{d}t^2\oplus f^2g_{2},$ where $f:I\rightarrow (0,\infty)$ is smooth, $I$ is an open interval in $\mathbb{R}.$
\end{Definition}

\begin{Proposition}
Let $M=I\times_{f}M_{2}$ be a generalized Robertson-Walker space-time with $\overline\nabla,\,P=\partial t\in\Gamma(TI),\,\zeta=a\partial t+\zeta_{2}\in\Gamma(TM),\,a\in\mathbb{R}\setminus\{0\}.$ Then for any
$X=u\partial t+X_{2}\in\Gamma(TM),\,u\in\mathbb{R}\setminus\{0\},$ we have:
$\zeta$ is a semi-symmetric metric Killing vector field on $M$ if $\zeta_{2}$ is a Killing vector field on $M_{2}$ and $f=e^t,\,g_{2}(X_{2},\zeta_{2})=0.$
\end{Proposition}
\begin{proof}
By Proposition 3.17.(3), we just need $a\partial t(f)+f\pi(a\partial t)=0,$ and
$\pi(u\partial t)g_{2}(X_{2},\zeta_{2})=0.$ \par  On the one hand,
$a\partial t(f)+f\pi(a\partial t)=a\dot{f}+fg(a\partial t,\partial t)=a\dot{f}-af=0.$ Consider that $a\neq0,$ we have $\dot{f}=f,$ then $f=e^t.$\par
On the other hand, $\pi(u\partial t)g_{2}(X_{2},\zeta_{2})=g(u\partial t,\partial t)g_{2}(X_{2},\zeta_{2})
=-ug_{2}(X_{2},\zeta_{2})=0,$ Consider that $u\neq0,$ we have $g_{2}(X_{2},\zeta_{2})=0.$
\end{proof}

By Corollary 3.16, we can easily get:
\begin{Proposition}
Let $M=M_{1}\times _{f}M_{2}$ be a warped product with $\overline\nabla,\,P\in\Gamma(TM_{2}),\,\zeta=\zeta_{1}+\zeta_{2}\in\Gamma(TM)$. Then $\zeta$ is a semi-symmetric metric Killing vector field if one of the following conditions holds:\\
$(1)\,\zeta=\zeta_{1},\,\zeta_{1}$ is a Killing vector field and $f\zeta_{1}(f)\|X_{2}\|_{2}^2-\pi(X_{2})g_{1}(X_{1},\zeta_{1})=0.$\\
$(2)\,\zeta=\zeta_{2},\,\zeta_{2}$ is a  Killing vector field and $\pi(\zeta_{2})\|X\|^2-f^2\pi(X_{2})g_{2}(X_{2},\zeta_{2})=0.$\\
$(3)\,\zeta=\zeta_{1}+\zeta_{2},$ each $\zeta_{i}$ is a Killing vector field and
$f\zeta_{1}(f)\|X_{2}\|_{2}^2+\pi(\zeta_{2})\|X\|^2-\pi(X_{2})g(X,\zeta)=0.$
\end{Proposition}

\begin{Proposition}
Let $M=M_{1}\times _{f}M_{2}$ be a warped product with $\overline\nabla,\,P\in\Gamma(TM_{2}),\,\zeta=\zeta_{1}+\zeta_{2}$ is a semi-symmetric metric Killing vector field, then:\\
$(1)\,\zeta_{1}$ is a Killing vector field on $M_{1}$ if $\pi(\zeta_{2})=0.$\\
$(2)\,\zeta_{2}$ is a Killing vector field on $M_{2}$ if $\pi(\zeta_{2})=0$ and
$f\zeta_{1}(f)\|X_{2}\|_{2}^2-\pi(X_{2})g(X,\zeta)=0.$
\end{Proposition}

Now we recall the definition of standard static space-times as follow:
\begin{Definition}
A {\bfseries standard static space-time} $(M,g)$ is a warped product of the form $M=M_{1}\times_{f}I$ with the metric tensor $g=g_{1}\oplus (-f^2)\mathrm{d}t^2,$ where $f:M_{1}\rightarrow (0,\infty)$ is smooth, $I$ is an open interval in $\mathbb{R}.$
\end{Definition}

\begin{Proposition}
Let $M=M_{1}\times_{f}I$ be a standard static space-time with $\overline\nabla,\,P=\partial t\in\Gamma(TI),\,\zeta=\zeta_{1}+a\partial t\in\Gamma(TM),\,a\in\mathbb{R}\setminus\{0\}.$ Then for any
$X=X_{1}+u\partial t\in\Gamma(TM),\,u\in\mathbb{R}\setminus\{0\},$ we have:
$\zeta$ is a semi-symmetric metric Killing vector field on $M$ if $\zeta_{1}$ is a Killing vector field on $M_{1}$ and $ufg_{1}(X_{1},\zeta_{1})-u^2\zeta_{1}(f)-af\|X_{1}\|_{1}^2=0.$
\end{Proposition}
\begin{proof}
By Proposition 3.21.(3), we just need
$$f\zeta_{1}(f)\|u\partial t\|_{2}^2+\pi(a\partial t)\|X_{1}+u\partial t\|^2
-\pi(u\partial t)g(X_{1}+u\partial t,\zeta_{1}+a\partial t)=0.$$
$$\Rightarrow f\zeta_{1}(f)g_{2}(u\partial t,u\partial t)
+g(a\partial t,\partial t)\Big[\|X_{1}\|_{1}^2+f^2g_{2}(u\partial t,u\partial t)\Big]
-g(u\partial t,\partial t)\Big[g_{1}(X_{1},\zeta_{1})+f^2g_{2}(u\partial t,a\partial t)\Big]=0,$$
$$\Rightarrow -u^2f\zeta_{1}(f)-af^2\Big[\|X_{1}\|_{1}^2-u^2f^2\Big]+uf^2\Big[g_{1}(X_{1},\zeta_{1})-auf^2\Big]=0,$$
$$\Rightarrow uf^2g_{1}(X_{1},\zeta_{1})-u^2f\zeta_{1}(f)-af^2\|X_{1}\|_{1}^2=0,$$
since $f>0,$ we get $ufg_{1}(X_{1},\zeta_{1})-u^2\zeta_{1}(f)-af\|X_{1}\|_{1}^2=0.$
\end{proof}

\section{\large Semi-symmetric Metric Killing Vector Fields on Multiply Warped Products with a Semi-symmetric Metric Connection}

\begin{Lemma}\label{l1}$^{[7]}$
Let $M=B\times _{f_{1}}M_{1}\times _{f_{2}}M_{2}\times\cdots\times _{f_{m}}M_{m}$ be a multiply warped product with $\overline\nabla$. If $X_{B},Y_{B}\in\Gamma(TB),\;X_{i},Y_{i}\in\Gamma(TM_{i})$ and $P\in\Gamma(TB),\,i\in\{1,\cdots,m\},$ then\\
$(1)\overline \nabla_{X_{B}}Y_{B}=\overline \nabla_{X_{B}}^{B}Y_{B}; $  \\
$(2)\overline \nabla_{X_{B}}Y_{i}=\frac{X_{B}(f_{i})}{f_{i}}Y_{i};$   \\
$(3)\overline \nabla_{Y_{i}}X_{B}=\Big[\frac{X_{B}(f_{i})}{f_{i}}+\pi(X_{B})\Big]Y_{i}; $   \\
$(4)\overline \nabla_{X_{i}}Y_{j}=0,$ if $i\neq j;$  \\
$(5)\overline \nabla_{X_{i}}Y_{i}=
-f_{i}g_{i}(X_{i},Y_{i})\mathrm{grad}f_{i}+\nabla_{X_{i}}^{i}Y_{i}-f_{i}^{2}g_{i}(X_{i},Y_{i})P.$
\end{Lemma}

\begin{Lemma}\label{l1}$^{[7]}$
Let $M=B\times _{f_{1}}M_{1}\times _{f_{2}}M_{2}\times\cdots\times _{f_{m}}M_{m}$ be a multiply warped product with $\overline\nabla$. If $X_{B},Y_{B}\in\Gamma(TB),\;X_{i},Y_{i}\in\Gamma(TM_{i})$ and $P\in\Gamma(TM_{r})$ for a fixed $r,\,i\in(1,\cdots,m),$ then\\
$(1)\overline \nabla_{X_{B}}Y_{B}=\nabla_{X_{B}}^{B}Y_{B}-g_{B}(X_{B},Y_{B})P; $  \\
$(2)\overline \nabla_{X_{B}}Y_{i}=\frac{X_{B}(f_{i})}{f_{i}}Y_{i}+\pi(Y_{i})X_{B};$   \\
$(3)\overline \nabla_{Y_{i}}X_{B}=\frac{X_{B}(f_{i})}{f_{i}}Y_{i}; $   \\
$(4)\overline \nabla_{X_{i}}Y_{j}=\pi(Y_{j})X_{i},$ if $i\neq j;$\\
$(4)\overline \nabla_{X_{i}}Y_{i}=
-f_{i}g_{i}(X_{i},Y_{i})\mathrm{grad}f_{i}+\nabla_{X_{i}}^{i}Y_{i}+\pi(Y_{i})X_{i}-f_{i}^{2}g_{i}(X_{i},Y_{i})P.$
\end{Lemma}

By Lemma 4.1 and straightforward computation, we can get:
\begin{Proposition}
Let $M=B\times _{f_{1}}M_{1}\times _{f_{2}}M_{2}\times\cdots\times _{f_{m}}M_{m}$ be a multiply warped product with $\overline\nabla,\,\zeta\in\Gamma(TM),\,P\in\Gamma(TB)$. Then for any vector fields $X,Y\in\Gamma(TM),$ we have:
\begin{eqnarray}
(\overline{\mathcal{L}}_{\zeta}g)(X,Y)&=&(\overline{\mathcal{L}}^{B}_{\zeta_{B}}g_{B})(X_{B},Y_{B})
+\sum\limits_{i=1}^{m}f_{i}^{2}(\mathcal{L}^{i}_{\zeta_{i}}g_{i})(X_{i},Y_{i})
+\sum\limits_{i=1}^{m}2[f_{i}\zeta_{B}(f_{i})+f_{i}^{2}\pi(\zeta_{B})]g_{i}(X_{i},Y_{i}) \nonumber \\
& &-\sum\limits_{i=1}^{m}f_{i}^{2}g_{i}(X_{i},\zeta_{i})\pi(Y_{B})
-\sum\limits_{i=1}^{m}f_{i}^{2}g_{i}(Y_{i},\zeta_{i})\pi(X_{B})
\end{eqnarray}
\end{Proposition}
By Lemma 4.2 and straightforward computation, we can get:
\begin{Proposition}
Let $M=B\times _{f_{1}}M_{1}\times _{f_{2}}M_{2}\times\cdots\times _{f_{m}}M_{m}$ be a multiply warped product with $\overline\nabla,\,\zeta\in\Gamma(TM),\,P\in\Gamma(TM_{r})$ for a fixed $r$. Then for any vector fields $X,Y\in\Gamma(TM),$ we have:
\begin{eqnarray}
(\overline{\mathcal{L}}_{\zeta}g)(X,Y)&=&(\mathcal{L}^{B}_{\zeta_{B}}g_{B})(X_{B},Y_{B})
+\sum\limits_{i=1}^{m}f_{i}^{2}(\mathcal{L}^{i}_{\zeta_{i}}g_{i})(X_{i},Y_{i})
+\sum\limits_{i=1}^{m}2\pi(\zeta_{i})g(X,Y)\nonumber \\
&+&\sum\limits_{i=1}^{m}2f_{i}\zeta_{B}(f_{i})g_{i}(X_{i},Y_{i})
-\sum\limits_{i=1}^{m}\pi(Y_{i})g(X,\zeta)-\sum\limits_{i=1}^{m}\pi(X_{i})g(Y,\zeta)
\end{eqnarray}
\end{Proposition}

Using the equation $(6)$ and Proposition 4.3, we have:

\begin{Corollary}
Let $M=B\times _{f_{1}}M_{1}\times _{f_{2}}M_{2}\times\cdots\times _{f_{m}}M_{m}$ be a multiply warped product with $\overline\nabla,\,P\in\Gamma(TB)$. Then for any vector field $X\in\Gamma(TM),$ we have:
\begin{eqnarray}
g(\overline\nabla_{X}\zeta,X)&=&g_{B}(\overline\nabla^B_{X_{B}}\zeta_{B},X_{B})
+\sum\limits_{i=1}^{m}f_{i}^{2}g_{i}(\nabla^i_{X_{i}}\zeta_{i},X_{i}) \nonumber \\
&-&\sum\limits_{i=1}^{m}f_{i}^{2}\pi(X_{B})g_{i}(X_{i},\zeta_{i})
+\sum\limits_{i=1}^{m}\big[f_{i}\zeta_{B}(f_{i})+f_{i}^{2}\pi(\zeta_{B})\big]\|X_{i}\|_{i}^2.
\end{eqnarray}
\end{Corollary}
Using the equation $(6)$ and Proposition 4.4, we have:
\begin{Corollary}
Let $M=B\times _{f_{1}}M_{1}\times _{f_{2}}M_{2}\times\cdots\times _{f_{m}}M_{m}$ be a multiply warped product with $\overline\nabla,\,\zeta\in\Gamma(TM),\,P\in\Gamma(TM_{r})$ for a fixed $r$. Then for any vector field $X\in\Gamma(TM),$ we have:
\begin{eqnarray}
g(\overline\nabla_{X}\zeta,X)&=&g_{B}(\nabla^B_{X_{B}}\zeta_{B},X_{B})
+\sum\limits_{i=1}^{m}f_{i}^{2}g_{i}(\nabla^i_{X_{i}}\zeta_{i},X_{i})
+\sum\limits_{i=1}^{m}\pi(\zeta_{i})\|X\|^2  \nonumber \\
&+&\sum\limits_{i=1}^{m}f_{i}\zeta_{B}(f_{i})\|X_{i}\|_{i}^2-\sum\limits_{i=1}^{m}\pi(X_{i})g(X,\zeta).
\end{eqnarray}
\end{Corollary}

By Corollary 4.5, we can easily get:
\begin{Proposition}
Let $M=B\times _{f_{1}}M_{1}\times _{f_{2}}M_{2}\times\cdots\times _{f_{m}}M_{m}$ be a multiply warped product with $\overline\nabla,\,P\in\Gamma(TB),\,\zeta=\zeta_{B}+\sum\limits_{i=1}^{m}\zeta_{i}\in\Gamma(TM)$. Then $\zeta$ is a semi-symmetric metric Killing vector field if one of the following conditions holds:\\
$(1)\,\zeta=\zeta_{B},\,\zeta_{B}$ is a semi-symmetric metric Killing vector field, and $\sum\limits_{i=1}^{m}\big[f_{i}\zeta_{B}(f_{i})+f_{i}^{2}\pi(\zeta_{B})\big]\|X_{i}\|_{i}^2=0.$\\
$(2)\,\zeta=\zeta_{i}$ for a fixed $i,\,\zeta_{i}$ is a  Killing vector field and $\pi(X_{B})g_{i}(X_{i},\zeta_{i})=0.$\\
$(3)\,\zeta=\zeta_{B}+\zeta_{i}$ for a fixed $i,\,\zeta_{B}$ is a semi-symmetric metric Killing vector field, $\zeta_{i}$ is a  Killing vector \\ \mbox{} \quad field, and
$\big[\zeta_{B}(f_{i})+f_{i}\pi(\zeta_{B})\big]\|X_{i}\|_{i}^2-f_{i}\pi(X_{B})g_{i}(X_{i},\zeta_{i})=0.$\\
$(4)\,\zeta=\sum\limits_{i=1}^{m}\zeta_{i},$ each $\zeta_{i}$ is a Killing vector field and  $\sum\limits_{i=1}^{m}f_{i}^2\pi(X_{B})g_{i}(X_{i},\zeta_{i})=0,$ for any $i\in\{1,\cdots,m\}.$\\
$(5)\,\zeta=\zeta_{B}+\sum\limits_{i=1}^{m}\zeta_{i},\,\zeta_{B}$ is a semi-symmetric metric Killing vector field, each $\zeta_{i}$ is a  Killing vector \\ \mbox{} \quad field and  $\sum\limits_{i=1}^{m}\big[f_{i}\zeta_{B}(f_{i})+f_{i}^2\pi(\zeta_{B})\big]\|X_{i}\|_{i}^2
-\sum\limits_{i=1}^{m}f_{i}^2\pi(X_{B})g_{i}(X_{i},\zeta_{i})=0,$ for any $i\in\{1,\cdots,m\}.$
\end{Proposition}

\begin{Proposition}
Let $M=B\times _{f_{1}}M_{1}\times _{f_{2}}M_{2}\times\cdots\times _{f_{m}}M_{m}$ be a multiply warped product with $\overline\nabla,\,P\in\Gamma(TB),\,\zeta=\zeta_{B}+\sum\limits_{i=1}^{m}\zeta_{i}$ is a semi-symmetric metric Killing vector field. Then:\\
$(1)\,\zeta_{B}$ is a semi-symmetric metric Killing vector field on $B.$\\
$(2)$ each $\zeta_{i}$ is a Killing vector field on $M_{i}$ if $\sum\limits_{i=1}^{m}\big[f_{i}\zeta_{B}(f_{i})+f_{i}^2\pi(\zeta_{B})\big]\|X_{i}\|_{i}^2=0.$
\end{Proposition}

By Corollary 4.6, we can easily get:
\begin{Proposition}
Let $M=B\times _{f_{1}}M_{1}\times _{f_{2}}M_{2}\times\cdots\times _{f_{m}}M_{m}$ be a multiply warped product with $\overline\nabla,\,P\in\Gamma(TM_{r})$ for a fixed $r,\,\zeta=\zeta_{B}+\sum\limits_{i=1}^{m}\zeta_{i}\in\Gamma(TM)$. Then $\zeta$ is a semi-symmetric metric Killing vector field if one of the following conditions holds:\\
$(1)\,\zeta=\zeta_{B},\,\zeta_{B}$ is a Killing vector field, and $\sum\limits_{i=1}^{m}\big[f_{i}\zeta_{B}(f_{i})\|X_{i}\|_{i}^2-\pi(X_{i})g_{B}(X_{B},Y_{B})\big]=0.$\\
$(2)\,\zeta=\zeta_{i}$ for a fixed $i\neq r,\,\zeta_{i}$ is a  Killing vector field ;\\
\mbox{}\;\;\;\,   $\zeta=\zeta_{i}$ for a fixed $i=r,\,\zeta_{i}$ is a  Killing vector field and
    $\pi(\zeta_{i})\|X_{i}\|_{i}^2-\pi(X_{i})g(X_{i},\zeta_{i})=0.$\\
$(3)\,\zeta=\zeta_{B}+\zeta_{i}$ for a fixed $i\neq r,\,\zeta_{B}$ is a Killing vector field, $\zeta_{i}$ is a  Killing vector field and $\zeta_{B}(f_{i})=0.$\\
\mbox{}\;\;\; $\,\zeta=\zeta_{B}+\zeta_{i}$ for a fixed $i=r,\,\zeta_{B}$ is a Killing vector field,
$\zeta_{i}$ is a  Killing vector field, and \\  \mbox{}\;\;\;\,
$\zeta_{B}(f_{i})\|X_{i}\|_{i}^2+f_{i}\pi(\zeta_{i})\|X_{i}\|_{i}^2-f_{i}\pi(X_{i})g_{i}(X_{i},\zeta_{i})=0.$\\
$(4)\,\zeta=\sum\limits_{i=1}^{m}\zeta_{i},$ each $\zeta_{i}$ is a Killing vector field, and  $\sum\limits_{i=1}^{m}f_{i}^2\big[\pi(\zeta_{i})\|X_{i}\|_{i}^2-\pi(X_{i})g_{i}(X_{i},\zeta_{i})\big]=0,$ for \\ \mbox{}\;\;\;\, $i\in\{1,\cdots,m\}.$\\
$(5)\,\zeta=\zeta_{B}+\sum\limits_{i=1}^{m}\zeta_{i},\,\zeta_{B}$ is a Killing vector field, each $\zeta_{i}$ is a  Killing vector field, and \\ \mbox{}\;\;\, $\sum\limits_{i=1}^{m}\big[f_{i}\zeta_{B}(f_{i})\|X_{i}\|_{i}^2+\pi(\zeta_{i})\|X\|^2-\pi(X_{i})g(X,\zeta)\big]=0.$
\end{Proposition}

\begin{Proposition}
Let $M=B\times _{f_{1}}M_{1}\times _{f_{2}}M_{2}\times\cdots\times _{f_{m}}M_{m}$ be a multiply warped product with $\overline\nabla,\,P\in\Gamma(TM_{r})$ for a fixed $r,\,\zeta=\zeta_{B}+\sum\limits_{i=1}^{m}\zeta_{i}$ is a semi-symmetric metric Killing vector field, then:\\
$(1)\,\zeta_{B}$ is a Killing vector field on $B$ if $\sum\limits_{i=1}^{m}\pi(\zeta_{i})=0.$\\
$(2)$ each $\zeta_{i}$ is a Killing vector field on $M_{i}$ if $\sum\limits_{i=1}^{m}\pi(\zeta_{i})=0$ and
$\sum\limits_{i=1}^{m}\big[f_{i}\zeta_{B}(f_{i})\|X_{i}\|_{i}^2-\pi(X_{i})g(X,\zeta)\big]=0.$
\end{Proposition}

\section{\large Killing Vector Fields on Multiply Warped Products}
Notice that when $P=0,$ the semi-symmetric metric Lie derivative becomes to Lie derivative, and the semi-symmetric metric Killing vector field becomes to Killing vector field. So the following four results are special cases of Section 4 when $P=0.$

\begin{Proposition}
Let $M=B\times _{f_{1}}M_{1}\times _{f_{2}}M_{2}\times\cdots\times _{f_{m}}M_{m}$ be a multiply warped product, $\zeta\in\Gamma(TM).$ Then for any vector fields $X,Y\in\Gamma(TM),$ we have:
$$(\mathcal{L}_{\zeta}g)(X,Y)=(\mathcal{L}^{B}_{\zeta_{B}}g_{B})(X_{B},Y_{B})
+\sum\limits_{i=1}^{m}f_{i}^{2}(\mathcal{L}^{i}_{\zeta_{i}}g_{i})(X_{i},Y_{i})
+2\sum\limits_{i=1}^{m}f_{i}\zeta_{B}(f_{i})g_{i}(X_{i},Y_{i}).  \eqno{(18)}$$
\end{Proposition}

\begin{Corollary}
Let $M=B\times _{f_{1}}M_{1}\times _{f_{2}}M_{2}\times\cdots\times _{f_{m}}M_{m}$ be a multiply warped product, $\zeta\in\Gamma(TM).$ Then for any vector field $X\in\Gamma(TM),$ we have:
$$g(\nabla_{X}\zeta,X)=g_{B}(\nabla^B_{X_{B}}\zeta_{B},X_{B})
+\sum\limits_{i=1}^{m}f_{i}^{2}g_{i}(\nabla^i_{X_{i}}\zeta_{i},X_{i})
+\sum\limits_{i=1}^{m}f_{i}\zeta_{B}(f_{i})\|X_{i}\|_{i}^2.   \eqno{(19)}$$
\end{Corollary}

\begin{Proposition}
Let $M=B\times _{f_{1}}M_{1}\times _{f_{2}}M_{2}\times\cdots\times _{f_{m}}M_{m}$ be a multiply warped product,$\,\zeta=\zeta_{B}+\sum\limits_{i=1}^{m}\zeta_{i}\in\Gamma(TM)$. Then $\zeta$ is a Killing vector field if one of the following conditions holds:\\
$(1)\,\zeta=\zeta_{B},\,\zeta_{B}$ is a Killing vector field, and $\sum\limits_{i=1}^{m}f_{i}\zeta_{B}(f_{i})\|X_{i}\|_{i}^2=0.$\\
$(2)\,\zeta=\zeta_{i}$ for a fixed $i,\,\zeta_{i}$ is a  Killing vector field.\\
$(3)\,\zeta=\zeta_{B}+\zeta_{i}$ for a fixed $i,\,\zeta_{B}$ is a Killing vector field, $\zeta_{i}$ is a Killing vector field, and $\zeta_{B}(f_{i})=0.$\\
$(4)\,\zeta=\sum\limits_{i=1}^{m}\zeta_{i},$ each $\zeta_{i}$ is a Killing vector field.\\
$(5)\,\zeta=\zeta_{B}+\sum\limits_{i=1}^{m}\zeta_{i},\,\zeta_{B}$ is a Killing vector field, each $\zeta_{i}$ is a  Killing vector field,$\,\sum\limits_{i=1}^{m}f_{i}\zeta_{B}(f_{i})\|X_{i}\|_{i}^2=0.$
\end{Proposition}

\begin{Proposition}
Let $M=B\times _{f_{1}}M_{1}\times _{f_{2}}M_{2}\times\cdots\times _{f_{m}}M_{m}$ be a multiply warped product, $\zeta=\zeta_{B}+\sum\limits_{i=1}^{m}\zeta_{i}$ is a Killing vector field. Then:\\
$(1)\,\zeta_{B}$ is a Killing vector field on $B.$\\
$(2)$ each $\zeta_{i}$ is a Killing vector field on $M_{i}$ if
$\sum\limits_{i=1}^{m}f_{i}\zeta_{B}(f_{i})\|X_{i}\|_{i}^2=0,$ for $i\in\{1,\cdots,m\}.$
\end{Proposition}

\section{\large 2-Killing Vector Fields on Multiply Warped Products}
\begin{Definition}\label{l1}$^{[2]}$
Let $M$ be a pseudo-Riemannian manifold, $\zeta\in\Gamma(TM)$ is called a {\bfseries 2-Killing vector field}, if
$$\mathcal{L}_{\zeta}\mathcal{L}_{\zeta}g=0  \eqno{(20)} $$
where $\mathcal{L}_{\zeta}$ is the Lie derivative in the direction of $\zeta$ on $M.$
\end{Definition}

\begin{Proposition}\label{l1}$^{[2]}$
Let $M$ be a pseudo-Riemannian manifold, $\zeta\in\Gamma(TM),$ then:
$$(\mathcal{L}_{\zeta}\mathcal{L}_{\zeta}g)(X,Y)=g(\nabla_{\zeta}\nabla_{X}\zeta-\nabla_{[\zeta,X]}\zeta,Y)
+g(X,\nabla_{\zeta}\nabla_{Y}\zeta-\nabla_{[\zeta,Y]}\zeta)+2g(\nabla_{X}\zeta,\nabla_{Y}\zeta)  \eqno{(21)} $$
for any vector fields $X,Y\in\Gamma(TM).$
\end{Proposition}

\begin{Corollary}\label{l1}$^{[2]}$
Let $M$ be a pseudo-Riemannian manifold, $\zeta\in\Gamma(TM),$ then $\zeta$ is a 2-Killing vector field if and only if:
$$R(\zeta,X,X,\zeta)=g(\nabla_{X}\zeta,\nabla_{X}\zeta)+g(\nabla_{X}\nabla_{\zeta}\zeta,X)  \eqno{(22)} $$
for any vector field $X\in\Gamma(TM).$
\end{Corollary}

\begin{Lemma}\label{l1}$^{[8]}$
Let $M$ be a pseudo-Riemannian manifold, $\zeta\in\Gamma(TM)$ is a Killing vector field of constant length, then
$\nabla_{\zeta}\zeta=0.$
\end{Lemma}

Then by Corollary 6.3 and Lemma 6.4, we have:
\begin{Corollary}
Let $M$ be a pseudo-Riemannian manifold, $\zeta\in\Gamma(TM),\,\zeta$ is a Killing vector field of constant length and also a 2-Killing vector field, then:
$$R(\zeta,X,X,\zeta)=g(\nabla_{X}\zeta,\nabla_{X}\zeta)\geq0  \eqno{(23)} $$
for any vector field $X\in\Gamma(TM).$
\end{Corollary}

\begin{Lemma}\label{l1}$^{[2]}$
Let $\zeta$ be a 2-Killing vector field on the compact $n-$dimensional pseudo-Riemannian manifold $(M,g)$ without boundary. If $Ric(\zeta,\zeta)\leq0,$ then $\zeta$ is a parallel vector field and $$Tr\Big(g(\nabla\zeta,\nabla\zeta)\Big)=0.   \eqno{(24)}$$
\end{Lemma}

\begin{Lemma}\label{l1}$^{[9]}$
Let $M=B\times _{f_{1}}M_{1}\times _{f_{2}}M_{2}\times\cdots\times _{f_{m}}M_{m}$ be a multiply warped product. If $X_{B},Y_{B}\in\Gamma(TB),\;X_{i},Y_{i}\in\Gamma(TM_{i}),\,i\in\{1,\cdots,m\},$ then\\
$(1)\nabla_{X_{B}}Y_{B}=\nabla_{X_{B}}^{B}Y_{B}; $  \\
$(2)\nabla_{X_{B}}Y_{i}=\nabla_{Y_{i}}X_{B}=\frac{X_{B}(f_{i})}{f_{i}}Y_{i};$   \\
$(4)\nabla_{X_{i}}Y_{j}=0,$  if $i\neq j;$\\
$(4)\nabla_{X_{i}}Y_{i}=\nabla_{X_{i}}^{i}Y_{i}-f_{i}g_{i}(X_{i},Y_{i})\mathrm{grad}f_{i}.$
\end{Lemma}

By Lemma 6.7 and straightforward computation, we can get:
\begin{Proposition}
Let $M=B\times _{f_{1}}M_{1}\times _{f_{2}}M_{2}\times\cdots\times _{f_{m}}M_{m}$ be a multiply warped product. $\zeta=\zeta_{B}+\sum\limits_{i=1}^{m}\zeta_{i}\in\Gamma(TM),$ then we have:
\begin{eqnarray} \setcounter{equation}{25}
(\mathcal{L}_{\zeta}\mathcal{L}_{\zeta}g)(X,Y)
&=&(\mathcal{L}^{B}_{\zeta_{B}}\mathcal{L}^{B}_{\zeta_{B}}g_{B})(X_{B},Y_{B})
+\sum\limits_{i=1}^{m}f_{i}^{2}(\mathcal{L}^{i}_{\zeta_{i}}\mathcal{L}^{i}_{\zeta_{i}}g_{i})(X_{i},Y_{i})
+4\sum\limits_{i=1}^{m}f_{i}\zeta_{B}(f_{i})(\mathcal{L}^{i}_{\zeta_{i}}g_{i})(X_{i},Y_{i})\nonumber \\
&+&2\sum\limits_{i=1}^{m}f_{i}\zeta_{B}\Big(\zeta_{B}(f_{i})\Big)g_{i}(X_{i},Y_{i})
+2\sum\limits_{i=1}^{m}\zeta_{B}(f_{i})\zeta_{B}(f_{i})g_{i}(X_{i},Y_{i})
\end{eqnarray}
\end{Proposition}

By Proposition 6.8, we have:
\begin{Corollary}
Let $M=B\times _{f_{1}}M_{1}\times _{f_{2}}M_{2}\times\cdots\times _{f_{m}}M_{m}$ be a multiply warped product. $\zeta=\zeta_{B}+\sum\limits_{i=1}^{m}\zeta_{i}\in\Gamma(TM),$ if $\zeta$ is a 2-Killing vector field, then:\\
$(1)\,\zeta_{B}$ is a 2-Killing vector field.\\
$(2)$ each $\zeta_{i}$ is a 2-Killing vector field if $\zeta_{B}(f_{i})=0,$ for $i\in\{1,\cdots,m\}.$
\end{Corollary}

\begin{Corollary}
Let $M=B\times _{f_{1}}M_{1}\times _{f_{2}}M_{2}\times\cdots\times _{f_{m}}M_{m}$ be a multiply warped product. $\zeta=\zeta_{B}+\sum\limits_{i=1}^{m}\zeta_{i}\in\Gamma(TM),$ if $\zeta_{B}$ and each $\zeta_{i}$ are 2-Killing vector fields on $B$ and $M_{i}$ respectively, then $\zeta$ is a 2-Killing vector field on $M$ if and only if one of the following conditions holds:\\
$(1)\,\zeta_{B}(f_{i})=0.$\\
$(2)$ each $\zeta_{i}$ is a homothetic vector field on $M_{i}$ with homothetic factor $c_{i}$(i.e, $\mathcal{L}^{i}_{\zeta_{i}}g_{i}=c_{i}g_{i}$) such that
$$f_{i}\zeta_{B}\Big(\zeta_{B}(f_{i})\Big)+\zeta_{B}(f_{i})\zeta_{B}(f_{i})=-2c_{i}f_{i}\zeta_{B}(f_{i}). \eqno{(26)}$$
for any $i\in\{1,\cdots,m\}.$
\end{Corollary}

\begin{Corollary}
Let $M=B\times _{f_{1}}M_{1}\times _{f_{2}}M_{2}\times\cdots\times _{f_{m}}M_{m}$ be a multiply warped product. $\zeta=\zeta_{B}+\sum\limits_{i=1}^{m}\zeta_{i}\in\Gamma(TM),$ then $\zeta$ is a 2-Killing vector field on $M$ if one of the following conditions holds:\\
$(1)\,\zeta_{B}$ and each $\zeta_{i}$ are 2-Killing vector fields on $B$ and $M_{i}$ respectively, and
$\zeta_{B}(f_{i})=0,$ for any \\ \mbox{}\quad\; $i\in\{1,\cdots,m\}.$\\
$(2)\,\zeta=\sum\limits_{i=1}^{m}\zeta_{i}$ and $\zeta_{i}$ is a 2-Killing vector field on $M_{i},$ for any $i\in\{1,\cdots,m\}.$
\end{Corollary}

\begin{Proposition}
Let $M=B\times _{f_{1}}M_{1}\times _{f_{2}}M_{2}\times\cdots\times _{f_{m}}M_{m}$ be a multiply warped product,
${\rm dim}B=n,\,{\rm dim}M_{i}=n_{i},\,i\in\{1,\cdots,m\},\,
\zeta=\zeta_{B}+\sum\limits_{i=1}^{m}\zeta_{i}\in\Gamma(TM),$ then
\begin{eqnarray}  \setcounter{equation}{27}
Tr\Big(g(\nabla\zeta,\nabla\zeta)\Big)&=&Tr\Big(g_{B}(\nabla^B\zeta_{B},\nabla^B\zeta_{B})\Big)
+\sum\limits_{i=1}^{m}Tr\Big(g_{i}(\nabla^i\zeta_{i},\nabla^i\zeta_{i})\Big)
+2\sum\limits_{i=1}^{m}\|\zeta_{i}\|_{i}^2\|\mathrm{grad}f_{i}\|_{B}^2 \nonumber \\
&+&\sum\limits_{i=1}^{m}\frac{n_{i}}{f_{i}^2}\Big(\zeta_{B}(f_{i})\Big)^2
+2\sum\limits_{i=1}^{m}\frac{\zeta_{B}(f_{i})}{f_{i}}\mathrm{div}_{i}\zeta_{i}.
\end{eqnarray}
\end{Proposition}
\begin{proof}
Suppose $\{e_{1}^B,e_{2}^B,\cdots,e_{n}^B\}$ is an orthonormal frame of $B,\,\{e_{1}^i,e_{2}^i,\cdots,e_{n_{i}}^i\}$ is an orthonormal frame of $M_{i},\,i\in\{1,\cdots,m\}.$ Then
$\{e_{1}^B,e_{2}^B,\cdots,e_{n}^B,e_{1}^1,e_{2}^1,\cdots,e_{n_{1}}^1,\cdots,e_{1}^m,e_{2}^m,\cdots,e_{n_{m}}^m\}$ is an orthonormal frame of $M.$ Then for any $\zeta\in\Gamma(TM),$ we have:
$$Tr\Big(g(\nabla\zeta,\nabla\zeta)\Big)
=\sum\limits_{i=1}^{n}g(\nabla_{e_{i}^B}\zeta,\nabla_{e_{i}^B}\zeta)
+\frac{1}{f_{1}^2}\sum\limits_{i=1}^{n_{1}}g(\nabla_{e_{i}^1}\zeta,\nabla_{e_{i}^1}\zeta)+\cdots
+\frac{1}{f_{m}^2}\sum\limits_{i=1}^{n_{m}}g(\nabla_{e_{i}^m}\zeta,\nabla_{e_{i}^m}\zeta).  $$
By Lemma 6.7 and straightforward computation, we get
$$\sum\limits_{i=1}^{n}g(\nabla_{e_{i}^B}\zeta,\nabla_{e_{i}^B}\zeta)=
Tr\Big(g_{B}(\nabla^B\zeta_{B},\nabla^B\zeta_{B})\Big)
+\sum\limits_{i=1}^{m}\|\zeta_{i}\|_{i}^2\|\mathrm{grad}f_{i}\|_{B}^2,$$
$$\frac{1}{f_{i}^2}\sum\limits_{j=1}^{n_{i}}g(\nabla_{e_{j}^i}\zeta,\nabla_{e_{j}^i}\zeta)
=\frac{n_{i}}{f_{i}^2}\Big(\zeta_{B}(f_{i})\Big)^2+Tr\Big(g_{i}(\nabla^i\zeta_{i},\nabla^i\zeta_{i})\Big)
+\|\zeta_{i}\|_{i}^2\|\mathrm{grad}f_{i}\|_{B}^2+2\frac{\zeta_{B}(f_{i})}{f_{i}}\mathrm{div}_{i}\zeta_{i}.$$
So we get the equation $(27).$
\end{proof}

\newtheorem{Theorem}[Definition]{Theorem}
\begin{Theorem}
Let $M=B\times _{f_{1}}M_{1}\times _{f_{2}}M_{2}\times\cdots\times _{f_{m}}M_{m}$ be a multiply warped product. $\zeta=\zeta_{B}+\sum\limits_{i=1}^{m}\zeta_{i}\in\Gamma(TM),\,B$ is a compact manifold without boundary, each $M_{i}$ is a compact manifold without boundary, then:\\
$(1)\,\zeta=\zeta_{B}+\sum\limits_{i=1}^{m}\zeta_{i}$ is parallel if $\zeta_{B}$ is a 2-Killing vector field, each $\zeta_{i}$ is a 2-Killing vector field,\\ \mbox{}  \;\, $\;Ric^B(\zeta_{B},\zeta_{B})\leq0,\,Ric^i(\zeta_{i},\zeta_{i})\leq0,$
and $f_{i}$ is constant, for $i\in\{1,\cdots,m\}.$\\
$(2)\,\zeta=\zeta_{B}$ is parallel if $\zeta_{B}$ is a 2-Killing vector field, $Ric^B(\zeta_{B},\zeta_{B})\leq0,$
and $\zeta_{B}(f_{i})=0.$\\
$(3)\,\zeta=\zeta_{B}+\zeta_{i}$ is parallel for a fixed $i$ if $\zeta_{B}$ is a 2-Killing vector field, $\zeta_{i}$ is a 2-Killing vector field,\\ \mbox{}  \;\, $\;Ric^B(\zeta_{B},\zeta_{B})\leq0,\,Ric^i(\zeta_{i},\zeta_{i})\leq0,$
and $f_{i}$ is constant.\\
$(4)\,\zeta=\zeta_{i}$ is parallel for a fixed $i$ if $\zeta_{i}$ is a 2-Killing vector field, $Ric^i(\zeta_{i},\zeta_{i})\leq0,$ and $f_{i}$ is constant.\\
$(5)\,\zeta=\sum\limits_{i=1}^{m}\zeta_{i}$ is parallel if each $\zeta_{i}$ is a 2-Killing vector field, $Ric^i(\zeta_{i},\zeta_{i})\leq0,$ and $f_{i}$ is constant, for \\ \mbox{}  \;\;\, $i\in\{1,\cdots,m\}.$
\end{Theorem}
\begin{proof}
$(1)$ Since $\zeta_{B}$ is a 2-Killing vector field, each $\zeta_{i}$ is a 2-Killing vector field, $\;Ric^B(\zeta_{B},\zeta_{B})\leq0,\,Ric^i(\zeta_{i},\zeta_{i})\leq0,\,B$ is a compact manifold without boundary, each $M_{i}$ is a compact manifold without boundary, by Lemma 6.6, we have $Tr\Big(g_{B}(\nabla^B\zeta_{B},\nabla^B\zeta_{B})\Big)=0,\,Tr\Big(g_{i}(\nabla^i\zeta_{i},\nabla^i\zeta_{i})\Big)=0.$ Then for constant function $f_{i},$ using Proposition 6.12, we get :
$$Tr\Big(g(\nabla\zeta,\nabla\zeta)\Big)=0.$$
Thus $\zeta$ is a parallel vector field with respect to the metric $g.$\par
Using the same proof method, we can easily get $(2),(3),(4),(5).$
\end{proof}

\begin{Theorem}
Let $M=B\times _{f_{1}}M_{1}\times _{f_{2}}M_{2}\times\cdots\times _{f_{m}}M_{m}$ be a multiply warped product. $\zeta$ is a non-trivial 2-Killing vector field on $M,$ let $K$ denote the sectional curvature, then:\\
$(1)$ If $\nabla_{\zeta}\zeta$ is parallel along a curve $\gamma,$ then $K(\zeta,\dot{\gamma})\geq0.$\\
$(2)$ If $\zeta$ is a Killing vector field of constant length, then $K(\zeta,\dot{\gamma})\geq0.$
\end{Theorem}
\begin{proof}
$(1)$ By Corollary 6.3, we have $R(\zeta,X,X,\zeta)=g(\nabla_{X}\zeta,\nabla_{X}\zeta)+g(\nabla_{X}\nabla_{\zeta}\zeta,X).$ Take $X=\dot{\gamma},$ since $\nabla_{\zeta}\zeta$ is parallel along a curve $\gamma,$ then we have $\nabla_{X}\nabla_{\zeta}\zeta=0,$ so
$$R(\zeta,X,X,\zeta)=g(\nabla_{X}\zeta,\nabla_{X}\zeta)=\|\nabla_{X}\zeta\|\geq0$$
$$\Rightarrow -R(\zeta,X,\zeta,X)=\|\nabla_{X}\zeta\|\geq0$$
$$\Rightarrow K(\zeta,\dot{\gamma})=-\frac{R(\zeta,X,\zeta,X)}{A^2(\zeta,\dot{\gamma})}\geq0,$$
where $A(\zeta,\dot{\gamma})$ is area of the parallelogram generated by $\zeta$ and $\dot{\gamma}.$\\
$(2)$ Take $X=\dot{\gamma},$ by Corollary 6.5, we have
$$R(\zeta,X,X,\zeta)=g(\nabla_{X}\zeta,\nabla_{X}\zeta)\geq0   $$
$$\Rightarrow -R(\zeta,X,\zeta,X)=\|\nabla_{X}\zeta\|\geq0$$
$$\Rightarrow K(\zeta,\dot{\gamma})=-\frac{R(\zeta,X,\zeta,X)}{A^2(\zeta,\dot{\gamma})}\geq0,$$
where $A(\zeta,\dot{\gamma})$ is area of the parallelogram generated by $\zeta$ and $\dot{\gamma}.$
\end{proof}


From [2], we can see that $(at-b)^{\frac{1}{3}}\partial t$ is a 2-Killing vector field on $I,$ where $a,b\in\mathbb{R},\,t\in I,$ and $t\neq\frac{b}{a}.$ Then by Corollary 6.10, we have:
\begin{Proposition}
Let $M=I\times _{f_{1}}M_{1}\times _{f_{2}}M_{2}\times\cdots\times _{f_{m}}M_{m}$ be a multiply warped product. $\zeta=(at-b)^{\frac{1}{3}}\partial t+\sum\limits_{i=1}^{m}\zeta_{i}\in\Gamma(TM).$ Suppose each $\zeta_{i}$ is a
2-Killing vector field on $M_{i},$ then $\zeta$ is a 2-Killing vector field on $M$ if each $\zeta_{i}$ is a homothetic vector field on $M_{i}$ with $c_{i}$ satisfying:
$$\frac{a}{3}f_{i}\dot{f_{i}}+(f_{i}\ddot{f_{i}}+\dot{f_{i}}^2)(at-b)=-2c_{i}f_{i}\dot{f_{i}}(at-b)^{\frac{2}{3}}.
\eqno{(28)}$$
\end{Proposition}

We recall the definition of generalized Kasner space-times in $[9].$
\begin{Definition}
A {\bfseries generalized Kasner space-time} $(M, g)$ is a Lorentzian multiply warped product of the form
$M=I \times_{\phi^{p_{1}}}M_{1} \times_{\phi^{p_{2}}}M_{2}\times\cdots \times_{\phi^{p_{m}}}M_{m}$ with the metric tensor $g=-dt^2\oplus \phi^{2p_{1}}g_{M_{1}}\oplus \phi^{2p_{2}}g_{M_{2}}\oplus\cdots\oplus \phi^{2p_{m}}g_{M_{m}},$
where $\phi:I\to(0,\infty)$ is smooth and $p_{i}\in\mathbb{R},$ for any $i\in\{1,\dots,m\}$ and also $I=(t_{1},t_{2}).$
\end{Definition}
Then consider that $f_{i}=\phi^{p_{i}},$ by Proposition 6.15, we get:
\begin{Proposition}
Let $M=I\times _{\phi^{p_{1}}}M_{1}\times _{\phi^{p_{2}}}M_{2}\times\cdots\times _{\phi^{p_{m}}}M_{m}$ be a generalized Kasner space-time. $\zeta=(at-b)^{\frac{1}{3}}\partial t+\sum\limits_{i=1}^{m}\zeta_{i}\in\Gamma(TM).$ Suppose each $\zeta_{i}$ is a
2-Killing vector field on $M_{i},$ then $\zeta$ is a 2-Killing vector field on $M$ if each $\zeta_{i}$ is a homothetic vector field on $M_{i}$ with $c_{i}$ satisfying:
$$\frac{a}{3}+\frac{2p_{i}-1}{\phi}(at-b)=-2c_{i}(at-b)^{\frac{2}{3}}.   \eqno{(29)}$$
\end{Proposition}

{\bfseries Acknowledgement.}
We would like to thank the referee for his(her) careful reading and helpful comments.


\clearpage

\begin{thebibliography}{999}
\bibitem{1}
  R. Bishop, B. O'Neill, Manifolds of negative curvature, Trans. Am. Math. Soc. 145(1969) 1-49.
\bibitem{2}
  T. oprea, 2-Killing vector fields on Riemannian manifolds, Balkan Journal of Geometry and Its Applications, 13(2008), No.1, 87-92.
\bibitem{3}
  S. Shenawy, B. \"Unal, 2-Killing vector fields on standard static space-times, arXiv:1411.6207.
\bibitem{4}
  H. A. Hayden, Subspace of a space with torsion, Proc. Lond. Math. Soc. 34(1932) 27-50.
\bibitem{5}
  K. Yano, On semi-symmetric metric connection, Rev. Roumaine Math. Pures Appl. 15(1970) 1579-1586.
\bibitem{6}
  S. Sular and C. \"Ozg\"ur, Warped products with a semi-symmetric metric connection, Taiwanese J. Math. 15 (2011), No. 4, 1701-1719.
\bibitem{7}
  Y. Wang, Multiply warped products with a semi-symmetric metric connection, Abstract and Applied Analysis. Vol 2014, Article ID 742371, 12pages.
\bibitem{8}
  Yu. G. Nikonorov, Killing vector fields of constant length on compact homogeneous Riemannian manifolds, arXiv:1504.03432.
\bibitem{9}
  F. Dobarro, B. \"Unal, Curvature of multiply warped products, J. Geom. Phys. 55(2005) 75-106.


\end{thebibliography}
\end{document}